\documentclass[12pt]{article}
\usepackage{amsmath, amssymb}
\numberwithin{equation}{section}
\usepackage{amsthm}
\usepackage{bbm}
\usepackage{hyperref}
\usepackage[a4paper, total={14.5cm, 23cm}]{geometry}
\newtheorem{thm}{Theorem}[section]
\newtheorem{lem}[thm]{Lemma}

\newtheorem{cor}[thm]{Corollary}
\newtheorem*{observ}{Observation}

\begin{document}
\title{\vspace{-10ex} Existence of equilibrium for Infinite System of Interacting Diffusions}
\author{Franti\v sek \v Z\' ak}
\date{Department of Mathematics, Imperial College London, f.zak12@imperial.ac.uk}
\maketitle
\begin{abstract}
\noindent We develop and implement new probabilistic strategy for proving basic results about long time behaviour for interacting diffusion processes on unbounded lattice. The concept of the solution used is rather weak as we construct the process as a solution to suitable infinite dimensional martingale problem. However the techniques allow us to consider cases where the generator of the particle is degenerate elliptic operator. As a model example we present situation, where the operator arises from Heisenberg group. In the last section we mention some further examples that can be handled using our methods. 
\end{abstract}
\begin{small}
\end{small}

\section{Introduction} The study of interacting particle systems has a long and profound history, as is well evidenced by excellent monographs \cite{Liget} or \cite{Kipnis}. Initially motivated by the problems of statistical physics, the field has grown into an important area of Markov processes in itself with interesting problems and rich interplay with other subjects. \\
We investigate continuous spin systems with a diffusion particle on each site. Most results establishing ergodicity properties for interacting particle systems with unbounded state space are tied with the use of functional inequalities, see \cite{Gionet}. As for the diffusions, there has been two independent successful approaches to this problem in the 1990s, one by Zegarli\'{n}ski \cite{ZegarS} and other by Da Prato and Zabczyk \cite{Zabczyk}, each to their merit and deficiencies. The approach in \cite{ZegarS} constructs the desired semigroup using finite dimensional approximations and ergodicity results are established via log Sobolev inequality, while more probabilistic approach in \cite{Zabczyk} uses the theory of SDEs on Hilbert spaces for construction and ergodicity is tied with the dissipativity properties of underlying operators. \\ 
Both these works essentially cover only elliptic case. The question how to address some subelliptic situation has been resolved under suitable condition in \cite{ZegarN} again using analytic techniques based on functional inequalities (very recently the results were extended to cover even broader class of operators in \cite{Ottobre} and \cite{Ottobrex}). Because in such cases even ergodicity of the finite system is highly non-trivial, important part of the result lies in conquering this problem. \\
This article presents a new probabilistic approach to investigate these issues. The results obtained go in some way successfully beyond Hilbert space methods of \cite{Zabczyk}.  We can cover degenerate multiplicative noise as we show in the case of Heisenberg group (or Grushin plane). However we cannot prove the uniqueness of invariant measure, let alone convergence towards it. Notice however that such results usually require some assumptions about smallness of the interactions, they should be tied with the condition on weights of the space the system live in, see \cite{Zabczyk} for example. Therefore it appears even probable, that under assumptions we work, the uniqueness of invariant measure for the system may not hold. \\
The setting is the following; assume we have a space $(\mathbb{R}^n)^{\mathbb{Z}^d}$, the dynamics of the system can then be described by the operator of the form \begin{align} \label{obecnyopera} \sum_{i \in \mathbb{Z}^d} \mathcal{A}_i + q_i \mathcal{B}_i, \end{align} where each $\mathcal{A}_i$ is the second order operator acting on $\mathbb{R}^n$ and on i-th coordinate of the lattice $\mathbb{Z}^d$ and $\mathcal{B}_i$ first order operator acting on i-th coordinate. We assume that we have interactions $q_i$ only in drift term and they are of finite range. \\
We construct the infinite dimensional process using finite dimensional approximations by solving appropriate stochastic differential equations. Of course such approach is well known and nothing new in the field, see e.g. \cite{Stroockinterakce}, \cite{Fritz}. The main novelty of our approach in comparison with these mentioned lies in two facts - we use martingale problem as a concept of solution, which allows us to bypass strong boundedness of coefficients assumption in \cite{Stroockinterakce}, secondly we benefit from now well established Meyn-Tweedie \cite{Mejn2} theory to prove ergodicity results in finite dimension. \\
In section 2 we give a proof of these finite dimensional results. Using tightness arguments we construct the process corresponding to (\ref{obecnyopera}) as a solution to martingale problem. The key and most technical part is section 4, where we show under additional assumptions about interaction functions that the limit of our approximations is unique and consequently establish Markov property of our process. The existence of the invaraint measure for the constructed process is established in the end. \\
For clarity and brevity of exposition we illustrate our techniques with the specific example of the operators corresponding to Heisenberg group. However it should be noted, that many parts of our results are independent of the specific diffusions considered, so in the last section we also mention some other natural situation that can be dealt with our methods. 
\noindent 
\subsection{Statement of the results and strategy of the proof} 
\noindent Let $\mathbb{H} = \mathbb{R}^3 = (x, y, z)$ be the Heisenberg group (for the detailed treatment of Heisenberg group as an example of Stratified Lie group see \cite{Grupy}, for nice and brief account of the relation to the matrix Heisenberg group see \cite{Bakry}) and $X, Y$ the generators of Lie algebra on $\mathbb{H}$, i. e. $$X = \partial_x - \frac{1}{2}y \partial_z$$ $$Y = \partial_y + \frac{1}{2}x \partial_z.$$ We denote $D = x \partial_x + y \partial_y + 2z \partial_z$ (so that $[X, D] = X$, $[Y,D] = Y$) the so-called dilation operator. \\
Consider the $d$ dimensional lattice $(\mathbb{R}^3)^{\mathbb{Z}^d}$, i. e. spin system where we have a copy of Heisenberg group at every point. We study the existence and long time behaviour of diffusion associated with the operator \begin{align} \label{opera} L = \sum_{i \in \mathbb{Z}^d} \mathcal{L}_{\lambda_i} +  q_{x_i} X_i + q_{y_i} Y_i, \end{align} where $X_{\cdot_i}$ is the vector field acting on the i-th coordinate, $q_{\cdot_i}$ is the interaction function with finite range, i. e. function whose value depends on all neighbours within some fixed length $r$, $\mathcal{L}_{\lambda_i} = X_i^2 + Y_i^2 - \lambda_i D_i$ and $\lambda_i$ are positive constants. We can summarize the results obtained as follows.
\begin{thm}[Inifinite dimensional results] \label{HlavniVysledek} 
Let $\mathbb{Z}^d$ be $d$ dimensional lattice equipped with the max metric, i. e. $\|i\|_{\max} = \max_{1 \leq j \leq d} \ |i_j|$ for $i \in \mathbb{Z}^d$, $r > 0$ given constant and $\Pi_n = \lbrace i \in \mathbb{Z}^d : \|i\|_{\text{max}} \leq nr \rbrace$. Let $q_{\cdot_i}, i \in \mathbb{Z}^d$ be smooth functions depending on $(2r+1)^d$ variables. Let $L$ be the operator given by $$L = \sum_{i \in \mathbb{Z}^d} \mathcal{L}_{\lambda_i} + q_{x_i} X_i + q_{y_i} Y_i$$ subject to the assumptions  : \begin{itemize}
\item \textbf{(H1)} $\exists C > 0 : \ \sup_{u \in (\mathbb{R}^3)^{(2r + 1)^d}} |q_{\cdot_i}(u)| \leq C, \ i \in \mathbb{Z}^d$
\item \textbf{(H2)} $\sup_{u \in (\mathbb{R}^3)^{(2r + 1)^d}} \sum_{j = 1}^{(2r+1)^d} |\frac{\partial q_{\cdot_i}}{\partial_j}(u) u_{\cdot_i} | +  |\frac{\partial q_{\cdot_i}}{\partial_j}(u)| \leq C, \ i \in \mathbb{Z}^d$
\item \textbf{(H3)} $\inf_{i \in \mathbb{Z}^d} \lambda_i > 0, \sup_{i \in \mathbb{Z}^d} \lambda_i < \infty.$
\end{itemize} Introduce the weighted metric space $$ S = \lbrace a \in (\mathbb{R}^3)^{\mathbb{Z}^d} : \ \sum_{i \in \mathbb{Z}^d} \|a_i\|_\mathbb{H}^8 u(i) < +\infty \rbrace,$$ where $\|a_i\|_\mathbb{H} = ((a_{i,x} + a_{i,y})^2 + a_{i,z}^2)^{\frac{1}{4}}$ and the weights satisfy : \begin{itemize}
\item \textbf{(H4)}$\sum_{i \in \mathbb{Z}^d} u(i) < + \infty, \ u(i) > 0, i \in \mathbb{Z}^d$
\item \textbf{(H5)} $\exists v(i) > 0, i \in \mathbb{Z}^d, \sum_i v(i) < +\infty, \sum_i \frac{u(i)}{v(i)} < +\infty$
\item \textbf{(H6)} $\exists \delta \in (0, 1)$ $\exists K > 0$ s. t. $$u(j) \geq \frac{K}{i!^{1 - \delta}} \  \ j \in \Pi_i \setminus \Pi_{i-1}, \ i \in \mathbb{N}.$$ 
\end{itemize} We naturally set $\|a\|_S = \sqrt[8]{\sum_{i \in \mathbb{Z}^d} \|a_i\|_{\mathbb{H}}^8 u(i)}$. \\
Then for any $a \in S$ there exists probability measure $P^a$ on $\Omega = C([0, \infty), S)$ (space of continuous functions form $[0,\infty)$ with values in $S$) such that for the canonical process $A_t(\omega) = \omega_t$ we have $P^a(A_0 = a) = 1$ and the process $$f(A_t) - f(A_0) - \int_0^t Lf(A_u) du$$ is martingale for $f \in C^{2, Cyl}_c$ under the measure $P^a$, where $C^{2, Cyl}_c$ stands for cylindrical twice continuously differentiable functions with compact support. The pair $(A_t, P^a)$ is a Markov process and there exist an invariant measure $\nu$ for the semigroup $P_t f(a) = E^a f(A(t)), a \in S$.
\end{thm}
\noindent The Theorem consists of several non-trivial ingredients, namely the existence of solution to the martingale problem is proved in Theorem \ref{existence reseni mart problemu}, Markov property in Theorem \ref{Markovsky proces} and the existence of invariant measure for the model is proved in Chapter \ref{ergodickakapitola}. \\
To reach these results, we firstly proceed by investigating the case of diffusion on Heisenberg group. Concretely we analyse the asymptotic behaviour of the Markov process on $\mathbb{R}^3$ with generator $$\mathcal{L} = X^2 + Y^2 - \lambda D + q_x X + q_y Y.$$ Under suitable assumptions on $q_\cdot'$s the process can be constructed by ordinary It\={o} stochastic equation and using the theory of Meyn and Tweedie (\cite{Mejn1}, \cite{Mejn2}, \cite{Hairer}) we establish exponential convergence in the total variation norm to the invariant measure in section 2. This result can be immediately translated to the exponential ergodicity of diffusion on $(\mathbb{R}^3)^n$ with the generator $$\sum_{i = 1}^n \mathcal{L}_{\lambda_i} + q_{x_i} X_i + q_{y_i} Y_i.$$ We prove in explicit the following result.\\
\begin{thm}[Finite Dimensional results] \label{konecnedimvysledky}
Let $(\mathbb{R}^3)^n$, $n \in \mathbb{N}$ be the state space and consider the operator \begin{align} \label{multidimoperator}
L_n = \sum_{i = 1}^n \mathcal{L}_{\lambda_i} + q_{x_i} X_i + q_{y_i} Y_i,
\end{align} under the corresponding assumptions \textbf{(H1), (H3)}. If we denote $A^n$ the diffusion corresponding to the operator (\ref{multidimoperator}), i. e. the unique solution to the It\={o} SDE with coefficients \begin{align*} 
b = & (q_{x_1} - \lambda_1 x_1, q_{y_1} - \lambda_1 y_1, - 2 \lambda_1 z_1 + \frac{1}{2} (q_{y_1} x_1 - q_{x_1} y_1), \ldots  \\
& \ldots, q_{x_n} - \lambda_n x_n, q_{y_n} - \lambda_n y_n, - 2 \lambda_n z_n + \frac{1}{2} (q_{y_n} x_n - q_{x_n} y_n) )
\end{align*}
\begin{align*}
\sigma = \begin{pmatrix}
M_1 & 0 & \cdots & 0 \\
0 & M_2 & \cdots & 0 \\
\vdots & 0 & \ddots & \vdots \\
0 & \cdots & 0 & M_n
\end{pmatrix}, \ \text{where} \ M_i = \begin{pmatrix} \sqrt{2} & 0 \\ 0 & \sqrt{2} \\ \frac{-y_i}{\sqrt{2}} & \frac{x_i}{\sqrt{2}}
\end{pmatrix},
\end{align*} then there exists unique invariant measure $\mu_n$ for the process $A^n$. For the function $W^k_n = 1 + \sum_{i = 1}^n v(i)((x_i^2 + y_i^2)^2 + z_i^2)^k$, $k \in \mathbb{N}$, where $v(i) > 0, \sum_i v(i) < \infty$, there exist constants $c_k > 0$ and $C_k > 0$ such that \begin{align} \label{obecnyLjapunov}
L_n W^k_n (a_n) \leq -c_k W^k_n(a_n) + C_k \ \forall a_n \in (\mathbb{R}^3)^n. 
\end{align} In addition there exist constant $K^k_n, \alpha^k_n > 0$, such that the following ($\mathcal{B}_b$ stands for bounded Borel functions)\begin{align} \label{kondimexponencialniergodicita} 
\sup_{\lbrace f \in \mathcal{B}(\mathbb{R}^{3n})  \ : \ \|f\|_\infty \leq 1 \rbrace} |E^a f(A^n(t)) - \mu_n(f) | \leq K^k_n W^k_n(a) e^{-\alpha^k_n t}
\end{align} holds for any $a \in (\mathbb{R}^3)^n$.
\end{thm} 
\noindent Next we consider an exhausting sequence $\Lambda_n \subset \subset \mathbb{Z}^d, \Lambda_n \nearrow \mathbb{Z}^d$ of the lattice, on $(\mathbb{R}^3)^{\Lambda_n}$ we construct diffusion $A^n$ that its generator extends the operator $$L_n = \sum_{i \in \Lambda_n} \mathcal{L}_{\lambda_i} + q_{x_i}^n X_i + q_{y_i}^n Y_i.$$ Unfortunately unlike in $\cite{Stroockinterakce}$ we are in a situation with unbounded coefficients, so we are unable to prove a limit of approximations in strong sense. Nevertheless we show tightness in appropriate weighted space $S$, $S \subset (\mathbb{R}^3)^{\mathbb{Z}^d}$, i. e. we are able to prove that the distributions of the processes $\tilde{A}^n = (A^n, 0_{i \in \mathbb{Z}^ \setminus \Lambda_n})$ form a tight sequence in $\Omega = C([0, \infty), S)$. From tightness follows the construction of family of measures $P^a, a \in S$ such that canonical process on $\Omega$ solves the martingale problem for (\ref{opera}). Our results are not completely satisfying since we do not prove the uniqueness of martingale problem for the operator (\ref{opera}). \\
Nevertheless under additional assumptions we can prove that our approximation procedure yields a unique measure. This is used to show that canonical process is a proper Markov process under constructed measure.  Furthermore exploiting the results obtained for bounded lattice we prove the existence of invariant measure for the unbounded lattice.\\
In certain aspects therefore - such as requiring no further assumptions on $\lambda$ in relevant examples - our results compare favourably to the ones in \cite{ZegarN}, \cite{Ottobre}. However it should be noted that our methods are only able to handle bounded interactions $q_\cdot'$s and we also work with much simpler generators than the authors in the above mentioned articles. One could also argue that our proofs are simpler, although that perhaps depends more on the background of the reader. 
\section{Finite dimensional case} 
\noindent We start by analyzing the diffusion on $\mathbb{R}^3$ associated with the second order operator  \begin{align} \label{1Doperator} \mathcal{L} = X^2 + Y^2 - \lambda D + q_x X + q_y Y. \end{align} We will work under the following assumptions \textbf{(B1)} : \begin{itemize}
\item $q_x, q_y \in C^\infty(\mathbb{R}^3, \mathbb{R}), \lambda > 0 $
\item $\exists \ C > 0 \ : \ ||q_x||_\infty \vee ||q_y||_\infty \leq C$
\end{itemize}
\noindent Under these assumptions we can construct the diffusion as a solution to the SDE $$dA(t) = b(A_t) dt + \sigma(A_t) dW_t.$$ Elementary computations with vector fields and matrices reveal that the coefficients can be chosen as 
\begin{equation}
\begin{aligned} \label{model}
b = (q_x - \lambda x, q_y - \lambda y, - 2 \lambda z + \frac{1}{2} (q_y x - q_x y) ) \\ 
\sigma = \begin{pmatrix} 
\sqrt{2} & 0 \\ 0 & \sqrt{2} \\ \frac{-y}{\sqrt{2}} & \frac{x}{\sqrt{2}}
\end{pmatrix}  \longrightarrow \frac{1}{2}a = \frac{1}{2} \sigma \sigma^* = \begin{pmatrix}
1 & 0 & \frac{-y}{2} \\
0 & 1 & \frac{x}{2} \\
\frac{-y}{2} & \frac{x}{2} & \frac{1}{4} (x^2 + y^2)
\end{pmatrix}. \end{aligned}
\end{equation}
The results of Meyn and Tweedie about exponential convergence of Markov processes can be stated in our diffusion context in the following way (for the  precise reference see \cite[Theorem 2.5]{Matingly} or very readable lecture notes by Rey-Bellet \cite{Bellet})
\begin{thm}[Harris - Meyn - Tweedie] \label{MejnTwidy} Let $X_t$ be a Markov process on $\mathbb{R}^n$ with transition probability $P_t$ and generator $L$. Suppose that following hypotheses are satisfied \begin{itemize}
\item[(1)] The Markov process is irreducible aperiodic, i. e. there exists $t_0$ (and then for all $t > t_0$) such that $$P_{t_0} (x, A) > 0, $$ for all $x \in \mathbb{R}^n$ and open sets $A$. 
\item[(2)] For any $t > 0$ the Markov semigroup $P_t(x, dy)$ has a density $p_t(x,y) dy$ which is a continuous function of $(x,y)$. 
\end{itemize}
Assume there exists Lyapunov function $$V : \mathbb{R}^n -> [1, \infty), \ V(x) \xrightarrow{\|x\| \rightarrow +\infty} + \infty$$ and constants $C, c > 0$ such that \begin{align} \label{Ljapunov} LV + cV \leq C. \end{align} Then there exists unique invariant measure $\mu$ for the process $X_t$ and there exist constants $K, \alpha > 0$ such that ($P_t f(a) = E^a f(X_t)$) $$\sup_{\lbrace f : |f(x)| \leq V(x) \rbrace} |E^a f(X_t) - \mu(f) | \leq K V(a) e^{-\alpha t}$$ for any $a \in \mathbb{R}^n.$
\end{thm}
\noindent Every verification of the stated result is non-trivial and depends on deep results about diffusions in $\mathbb{R}^n$. In the remainder of the section we show that the process $A$ given by SDE with the coefficients $(\ref{model})$ indeed satisfies the condition of the above theorem. The existence and smoothness of transition probability density is the immediate consequence of the H\"{o}rmander theorem in probabilistic settings. The version that is sufficient for our purposes was first established following H\"{o}rmander work in \cite{Japonci}.
\begin{thm}[H\"{o}rmander probabilistic setting, Ichihara - Kunita] \label{HormavPrav}
Assume $X_t$ is the unique strong solution to the Stratonovich SDE $$dX_t = b(X_t) dt + \sum_{i = 1}^d \sigma(X_t) \circ dW_t, $$ where $b, \sigma_i, 1 \leq i \leq d \ \in C^\infty(\mathbb{R}^n, \mathbb{R})$. Suppose that following (H\"{o}rmander) condition is satisfied $$\ \text{dim} \ (\text{Lie} \lbrace \sigma_1, \ldots, \sigma_d \rbrace) = n \ \forall x \in \mathbb{R}^n.$$ Then there exists probability density function $P_t(x, dy) = p_t(x, y) dy$ such that $p_t(x,y) \in C^\infty ( (0, \infty), \mathbb{R}^n, \mathbb{R}^n)$.
\end{thm}
\noindent In our case (\ref{model}) the drift in the Stratonovich form is actually the same as in It\={o} form. In any case the Lie algebra generated by the diffusion satisfies the H\"{o}rmander condition as elementary computation reveals that $[X,Y] = \partial_z$ and consequently $$\text{dim} \left(\text{Lie} \left\{ \left(\sqrt{2}, 0, \frac{-y}{\sqrt{2}}\right), \left(0, \sqrt{2}, \frac{x}{\sqrt{2}}\right) \right\} \right) = 3,$$ thus according to the above cited theorem we have the smoothness of transition probability density for (\ref{model}). \\
To investigate the irreducibility of diffusion, we would like to use Stroock - Varadhan support theorem (\cite{Stroock}), so the question is whether that we can solve the corresponding control problem. The version that accounts for unbounded coefficients we use, was proved in \cite{Madari}. \\
Let $F$ be the subset of the absolutely continuous functions $u : [0, t] \rightarrow \mathbb{R}^d$ with $u(0) = 0$ such that $F$ contains every infinitely differentiable function form $[0,t]$ to $\mathbb{R}^d$ vanishing at zero.  For the ordinary differential equation 
\begin{equation} \begin{aligned} \label{kontrola}
& \dot{x}^{u}(t) = b(x^{u}(t)) + \sum_{i=1}^d \dot{u}_i(t) \sigma_i (x^u(t)) \\ 
& x^{u}(0) = x_0 \ \in \mathbb{R}^n \end{aligned} \end{equation}
we denote $\mathcal{O}(t,x_0) = \lbrace y \in \mathbb{R}^n : x^u(t) = y, u \in F \rbrace$.
\begin{thm}[Stroock - Varadhan support theorem, \cite{Madari}] \label{Suport}
Let $X_t$ be the solution to the Stratonovich SDE \begin{align} \label{SSDE} dX_t = b(X_t) dt + \sum_{i = 1}^d \sigma_i \circ dW, \ X(0) = x, \end{align} where the coefficients satisfy linear growth assumptions, $b$ is Lipschitz and $\sigma_i, 1 \leq i \leq d$ are smooth with bounded derivatives. Let $P_t$ be the transition probability function related to (\ref{SSDE}) and $\mathcal{O}(t,x)$ be the orbit to the corresponding equation (\ref{kontrola}). Then $\text{supp} \ P_t(x, \cdot) = \overline{\mathcal{O}(t,x)}$.
\end{thm}

\begin{lem} \label{suportjeR}
Let $P_t$ be the transition function for the equation (\ref{model}). Then $\text{supp} \ P_t(x, \cdot) = \mathbb{R}^3$ for any $t > 0$ and $x \in \mathbb{R}^3$.
\end{lem}
\begin{proof}
We make of use the classical Girsanov transform \cite[pp. 166]{Oksendal} to simplify the control problem. Concretely the statement that the support of  diffusions $X_t, Y_t$ $$dX_t = b(X) dt + \sigma(X) dW$$ \begin{align} \label{transofrmace} dY_t = \tilde{b}(Y) dt + \sigma(Y) dW, \end{align} where $\sigma$ and $b$ are as in (\ref{model}) and $$\tilde{b} = (-\lambda x, -\lambda y, - 2 \lambda z)$$ is the same, provided we can find such $u : \mathbb{R}^3 \rightarrow \mathbb{R}^{2 \times 1} $ that $$\sigma u = b - \tilde{b}.$$ However it is easy, since $b - \tilde{b} = (q_x, q_y, \frac{1}{2}(q_y x - q_x y))$ and hence $$ \begin{pmatrix} 
\sqrt{2} & 0 \\ 0 & \sqrt{2} \\ \frac{-y}{\sqrt{2}} & \frac{x}{\sqrt{2}}
\end{pmatrix}  \begin{pmatrix}
\frac{q_x}{\sqrt{2}} \\ \frac{q_y}{\sqrt{2}}
\end{pmatrix} = \begin{pmatrix}
q_x \\ q_y \\ \frac{1}{2}(q_y x - q_x y)
\end{pmatrix} .$$ Hence to establish the theorem it suffices to prove the irreducibility of transition function corresponding to (\ref{transofrmace}).
Since the equation (\ref{transofrmace}) satisfies the Theorem \ref{Suport}, we only need to prove controllability of the system  
\begin{equation}
\begin{aligned} \label{zkontorluj}
& \dot{x} = \sqrt{2}\dot{u}_1 - \lambda x \\  
& \dot{y} = \sqrt{2}\dot{u}_2 - \lambda y \\
& \dot{z} = -\frac{y}{\sqrt{2}} \dot{u}_1 + \frac{x}{\sqrt{2}}\dot{u}_2 - 2 \lambda z
\end{aligned} \end{equation}  for $u \in H$, i. e. to show that from any starting point $(x_0, y_0, z_0)$ we can choose such $u \in H$ that $x(t) = x_t, y(t) = y_t, z(t) = z_t$, where $(x_t, y_t, z_t) \in \mathbb{R}^3$ are prescribed ending points. If we simply choose control $\dot{u}_1 (s) = a s + b$, $\dot{u}_2(s) = c s +d, $ then the problem (\ref{zkontorluj}) is reduced to solving three linear equations with four parameters, so the Lemma is proved. 
\end{proof}
\noindent The proof of existence of Lyapunov function for the operator (\ref{1Doperator}) satisfying (\ref{Ljapunov}) is elementary, albeit bit tedious.
\begin{lem} \label{VjeLjapunov}
Let $\mathcal{L}$ be the operator defined by (\ref{1Doperator}) under the assumptions \textbf{(B1)}. For the function $V^k = ((x^2 + y^2)^2 + z^2)^k$, $k \in \mathbb{N}$, there exist constants $c_k, C_k > 0$ such that \begin{align} \label{ljapunovskanerv} \mathcal{L} V^k + c_k V^k \leq C_k \ \forall (x, y, z) \in \mathbb{R}^3.\end{align} 
\end{lem}
\begin{proof}
We first compute the case for $V^k$, $k = 1$ (and omit the index in such case) and then proceed to general $k$. Using that $V_{xz} = V_{yz} = 0$ we calculate \begin{equation} \begin{aligned} \label{Vodhady} &\mathcal{L}V + cV = V_x (q_x - \lambda x) + V_y (q_y - \lambda y) + V_z (-2\lambda z + \frac{1}{2} (q_y x - q_x y)) \\
&+ V_{xx} + V_{yy} + V_{zz} \frac{1}{4}(x^2 + y^2) = 4x(x^2 + y^2)(q_x - \lambda x)\\ & + 4 y (x^2 + y^2)(q_y - \lambda y) + 2z(-2 \lambda z + \frac{1}{2}(q_y x - q_x y)) + 16 (x^2 + y^2)\\ & + \frac{1}{2} (x^2 + y^2) + cx^4 + cy^4 + 2x^2y^2c + cz^2 \\ & \leq (x^4 +y^4 + z^2 + 2x^2y^2)(c-4\lambda) + o(x^4) + o(y^4) + o(z^2). \end{aligned} \end{equation}
To obtain last inequality one uses bounds for $q_\cdot$'s and then Young inequality, e. g.  $|Czx| \lesssim |z|^{\frac{3}{2}} + |x|^{3}$(we use $\lesssim$ throughout the paper to denote the statement $A \lesssim B \Longleftrightarrow \ \exists C > 0: \ A \leq CB$). The resulted inequality obviously implies that for any $\lambda > 0$ we can choose $c > 0$ in such way, that $\mathcal{L}V + cV$ is bounded. 
For $V^k$ we get \begin{equation} \begin{aligned} \label{Vkodhady} & \mathcal{L}V^k + c_kV^k = V^{k-2}k \big(V V_x (q_x - \lambda x) + VV_y(q_y - \lambda y) \\ & +V V_z (-2\lambda z + \frac{1}{2}(q_y x - q_x y)) + (k-1)V_x^2 + V V_{xx} \\ & +(k-1)V_y^2 + VV_{yy} + \frac{1}{4}(x^2+y^2)(VV_{zz} + (k-1) V_z^2) \\ & - y (k-1)V_xV_z + x(k-1)V_yV_z +\frac{c_k}{k} V^2 \big). \end{aligned}
\end{equation} In a similar manner as we obtained (\ref{Vodhady}), (\ref{Vkodhady}) can be estimated as \begin{align*} &
\mathcal{L}V^k + c_kV^k \leq V^{k-2}k \big((\frac{c_k}{k} - 4 \lambda)(x^8 + y^8 +z^4) \\& + x^4y^4(\frac{4c_k}{k}- 8\lambda) + o(x^8) + o(y^8) + o(z^4)\big).
\end{align*} Notice that we not only proved boundedness of $\mathcal{L}V^k + c_kV^k$, but even obtained $$\mathcal{L}V^k + c_kV^k \xrightarrow{\|(x,y,z)\| \rightarrow + \infty} - \infty.$$
\end{proof} 
\noindent The Meyn - Tweedie theory as stated in Theorem \ref{MejnTwidy} now ensures exponential convergence to equilibrium for diffusion corresponding to the operator (\ref{1Doperator}). 
\begin{proof}[Proof of Theorem \ref{konecnedimvysledky}]
The proof differs only slightly from what we just showed for the case of $\mathbb{R}^3$ thanks to the structure of our coefficients $b$ and $\sigma$ and assumptions \textbf{(H1), (H3)}. As for the support problem, the situation is pretty much the same as in Lemma \ref{suportjeR}, and the smoothness of transition probability follows again immediately from H\"{o}rmander type theorem \ref{HormavPrav}. \\
It remains to show that for $W^k$ there exist constants $c_k, C_k > 0$ such that $$L_n W^k_n + c_k W^k_n \leq C_k$$ holds uniformly regardless of $n$. By denoting $V^k_i = ((x^2_i + y^2_i)^2 + z_i^2)^k$ and $\mathcal{L}_i = X_i^2 + Y^2_i - \lambda_i D_i + q_{x_i} X_i + q_{y_i} Y_i$, we can write \begin{align} \label{Wvypocet} L_n W^k_n + c_k W^k_n = c_k + \sum_{i=1}^n u(i) (\mathcal{L}_i V^k_i + c_k V^k_i).\end{align} The analysis of expression $\mathcal{L}_i V^k_i + c_k V^k_i$ was done in previous Lemma \ref{VjeLjapunov}. Notice that thanks to the assumption $\inf_{i \in \mathbb{Z}^d} \lambda_i > 0$ and the fact that bound for $q_\cdot$'s is uniform, the $c_k$ can be chosen in such way, that the following is true \begin{align*}
\exists \tilde{C}_k >0 \ : \forall_{1 \leq i \leq n} \ \ \mathcal{L}_i V^k_i + c_k V_i^k \leq \tilde{C}_k.
\end{align*} We install this into (\ref{Wvypocet}) and using hypothesis $\sum_{i = 1}^\infty u(i) < \infty$ we infer the desired bounds \begin{align*}& L_n W^k + c_k W^k \leq c_k + \sum_{i = 1}^n u(i) \tilde{C}_k \\ & \leq c_k + \tilde{C}_k \sum_{i = 1}^\infty u(i) := C_k < + \infty.
\end{align*} Hence indeed the Theorem \ref{MejnTwidy} can be applied  to prove the statement.
\end{proof}
\noindent \textbf{Remark.} It should be noted that the constants in the formula (\ref{kondimexponencialniergodicita}) cannot be chosen in such a way, that they would be independent of the dimension, even though the constants in (\ref{obecnyLjapunov}) are. It cannot be expected that one could prove the convergence in the total variation norm in the infinite dimension for interacting particle system. Let us add a simple reason for this fact. 
\begin{observ}
Let $\varpi$ and $\varrho$ be two probability measures on $\mathbb{R}$, such that $\mu \neq \nu$. Then $\|\varpi^n - \varrho^n\|_{TV} \xrightarrow{n \rightarrow \infty} 1$, where $\|\cdot\|_{TV}$ means total variation norm.
\end{observ}
\begin{proof}
As $\varpi \neq \varrho$, there exists $f \in C_b(\mathbb{R})$, such that $\varpi f \neq \varrho f$. \\ Say $\epsilon = |\varpi f - \varrho f| >0$. Define sets $A_n = \lbrace x \in \mathbb{R}^n : |\frac{1}{n} \sum_{i = 1}^n f(x_i) - \varpi f| < \frac{\epsilon}{2} \rbrace$. Weak Law of large numbers asserts $\varpi^n(A_n) \rightarrow 1$, while $\varrho^n(A_n) \rightarrow 0.$
\end{proof} 
\noindent Therefore even for the system without any interactions, one cannot have the constants independent of the dimension, unless the convergence to the invariant measure happens in finite amount of time. 
\section{Construction of the infinite dimensional measure}
There are several papers dealing with infinite dimensional martingale problems (\cite{Bass1}, \cite{Bass2}, \cite{Zamboti}) that establishes uniqueness as well, but all are based in elliptic settings and none can be directly applied to our case.  \\
The following version of Arzel\` a - Ascoli theorem follows easily from the general version proved in \cite[Theorem 47.1]{Munkres}.
\begin{thm}[Arzel\` a - Ascoli] \label{A-A}
Let $Y$ be a complete metric space and $f_n \in C([0, \infty), Y)$ sequence of equicontinuous functions. Endow $C([0, \infty), Y)$ with the topology of uniform convergence on compacts. If $\lbrace f_n(t) \rbrace$ is precompact in $Y$ on a dense set of $t \in [0, \infty)$, then $\lbrace f_n \rbrace$ is precompact in $C([0, \infty), Y)$. 
\end{thm}
\noindent To prove equicontinuity we use a variant of Kolmogorov continuity theorem (see \cite[chap. 8]{BassSP} for details).
\begin{thm} \label{stejnspoj}
Let $X^n$ be continuous processes taking values in some metric space $(S, \rho)$. Suppose for any $T > 0$ there exists constants $C(T), \epsilon > 0$ and $p > 0$ such that $$\sup_n E \ \rho(X^n_s, X^n_t)^p \leq C(T) |t  - s|^{1 + \epsilon} \ \  0 \leq s \leq t \leq T. $$ Then $\lbrace X^n \rbrace$ is equicontinuous family of processes with probability $1$.  
\end{thm}
\noindent The space on which we construct our measure is dictated to us by our Lyapunov function for (\ref{1Doperator}), so that we will be able to utilize the uniform bound (\ref{kondimexponencialniergodicita}). However we also have to choose space such that the Theorem \ref{stejnspoj} will be satisfied. For the sake of completeness let us clarify, that function of $V$ type indeed equips $\mathbb{R}^3$ with the metric. 
\begin{lem}
Endow $\mathbb{R}^3$ with the following operation $d$ : $$d(a,b) = \sqrt[4]{((a_x - b_x)^2 + (a_y - b_y)^2)^2 + (a_z - b_z)^2}. $$ $(\mathbb{R}^3, d)$ is then a metric space.
\end{lem}
\begin{proof}
The only non-trivial part is the triangle inequality. Hence we want to prove \begin{equation}  \begin {aligned} \label{trojnerovnost}
&\sqrt[4]{((a_x - b_x)^2 +  (a_y -  b_y)^2)^2 + (a_z - b_z)^2}& \leq \\
& \sqrt[4]{((a_x - c_x)^2 + (a_y - c_y)^2)^2 + (a_z - c_z)^2}   \\ 
& +\sqrt[4]{((c_x - b_x)^2 + (c_y - b_y)^2)^2 + (c_z - b_z)^2}.
\end{aligned}
\end{equation}
Notice that (\ref{trojnerovnost}) is clearly valid if either terms on $z$ axis are zero, or both $x$ and $y$ terms are zero. Therefore it remains to prove that if for $A, B, C, D, E, F \geq 0$ \begin{equation} \begin{aligned} \label{Minkowski} \sqrt[4]{A} \leq \sqrt[4]{B} + \sqrt[4]{C} \\
 \sqrt[4]{D} \leq \sqrt[4]{E} + \sqrt[4]{F},
\end{aligned} \end{equation}
then \begin{align} \label{finalnerovnost} \sqrt[4]{A + D} \leq \sqrt[4]{B + E} + \sqrt[4]{C + F}. \end{align}
The left side in (\ref{finalnerovnost}) is clearly maximized, if the left sides in (\ref{Minkowski}) is maximized. This happens, if we have equality in (\ref{Minkowski}). Hence it suffices to prove $$\sqrt[4]{(\sqrt[4]{B} + \sqrt[4]{C})^4 + (\sqrt[4]{E} + \sqrt[4]{F})^4} \leq  \sqrt[4]{B + E} + \sqrt[4]{C + F}, $$ but this follows from ordinary Minkowski inequality for $4$ - norm on $\mathbb{R}^2$. 
\end{proof}
\noindent We will denote by $\| \cdot \|_{\mathbb{H}}$ the function that assigns to $a \in \mathbb{R}^3$ value corresponding to the metric just defined, so that $\|a\|_{\mathbb{H}} = \sqrt[4]{((a_x^2 + a_y^2)^2 + a_z^2}.$ Given $d$ dimensional lattice $\mathbb{Z}^d$, we introduce the weighted space $$S = \lbrace a \in \mathbb{H}^{\mathbb{Z}^d} : \ \sum_{i \in \mathbb{Z}^d} \|a_i\|^8_{\mathbb{H}} u(i) < + \infty \rbrace.$$ For now it suffices to assume about the weights \textbf{(H4)} \begin{itemize}
\item 
$ \sum_{i \in \mathbb{Z}^d} u(i) < + \infty , \ u(i) > 0.$
\end{itemize} From the Lemma above we can infer following usual considerations that $S$ with the metric $\|a - b\|_S = \sqrt[8]{\sum_{i \in \mathbb{Z}^d} \|a_i - b_i\|_{\mathbb{H}}^8 u(i)}$, $a, b \in S$ is complete separable metric space and so consequently $\Omega = C([0, \infty), S)$ is Polish too. Let us describe compact sets of $S$. 
\begin{lem} \label{kompaktnostvE} Let $M \subset S$. Assume that $M$ is bounded and the following condition $$\forall \epsilon > 0 \  \exists n_0 \in \mathbb{N} \ \forall a \in M: \sum_{i = n_0}^{\infty} \|a_i\|_{\mathbb{H}}^8 u(i) < \epsilon.$$ Then $M$ is precompact in $S$.
\end{lem}
\begin{proof}
We show that from any sequence $\lbrace a^n \rbrace$ one can extract a Cauchy
sequence. By assumptions for a given $\epsilon > 0$ we find $n_0$, so we control the rest
of the sequence, and on the first $n_0 - 1$ coordinates simply choose a Cauchy
sequence step by step, which is possible by the boundedness assumption.
\end{proof}
\subsection{Moments estimates and tightness of approximations}
Let $\Lambda_n, |\Lambda_n| = N < + \infty$ be the exhausting sequence of $\mathbb{Z}^d$, i. e. $\Lambda_{n+1} \supseteq \Lambda_n$, $\bigcup_n \Lambda_n = \mathbb{Z}^d.$ 
We wish to construct martingale solution for the operator \begin{align} \label{infdimopera}
L = \sum_{i \in \mathbb{Z}^d} \mathcal{L}_{\lambda_i} + q_{x_i} X_i + q_{y_i} Y_i \ .
\end{align}
Suppose we have maximum metric on $\mathbb{Z}^d$ and we assume there exists constant $r > 0$ such that $q_{\cdot_i}$ depends only on neighbours within distance $r$. More precisely we assume about interaction functions $q$'s \textbf{(H1)} : \begin{itemize}
\item $ q_{\cdot_i} \in C^\infty((\mathbb{R}^3)^{\Pi_i}, \mathbb{R}), \ \text{where} \ \Pi_i = \lbrace j \in \mathbb{Z}^d : |j - i|_{\max} \leq r \rbrace$
\item $\exists C > 0 \ : \ \|q_{\cdot_i} \|< C.$ \end{itemize} 
About constants $\lambda_i$ we assume \textbf{(H3)} : 
\begin{itemize}
\item $\inf_{i \in \mathbb{Z}^d} \lambda_i > 0, \ \sup_{i \in \mathbb{Z}^d} \lambda_i < +\infty$.
\end{itemize}
On each space $(\mathbb{R}^3)^N$ we consider diffusion $A^n$ with generator that coincides on $C^2_c((\mathbb{R}^3)^{\Lambda_n})$ functions with \begin{align} \label{aproxopera}
L_n = \sum_{i \in \Lambda_n} L_{\lambda_i} + q_{x_i}^n X_i + q_{y_i}^n Y_i .
\end{align}
The interaction functions $q_{\cdot_i}$ in general depend on $n$, but in case point $i \in \mathbb{Z}^d$ has all neighbours in distance $r$, we put $q_{\cdot_i}^n = q_{\cdot_i}$, otherwise the functions have to be redefined, but we keep their smoothness and boundedness by $C$, so that they obey \textbf{(H1)}. \\
Set $\tilde{A}^n = (A^n, 0_{ i \in \mathbb{Z}^d \setminus \Lambda_n })$, then each $\tilde{A}^n(t)$ has values in $S$ and therefore $\tilde{A}^n$ lives in $\Omega = C([0, \infty), S)$.
\begin{lem} \label{momentoveodhady}
Let $a \in S$. For $n \in \mathbb{N}$ define $A^n$ as above with initial condition  $A^n(0) = \pi_{\Lambda_n} (a)$ and subsequently define $\tilde{A}^n$. Assume \textbf{(H1), (H3), (H4)}.  Let $T > 0$ be given, then there exist constants $C(T) > 0$  \begin{align} \label{spojitostodhad}
& \sup_n \ \forall_{0 \leq s \leq t \leq T} \ E \|\tilde{A}^n(t) - \tilde{A}^n(s)\|_S^8 \leq C(T) |t - s|^2 \\
\label{kompaktodhad} & \forall \delta >0 \ \forall t \geq 0 \ \exists N_0(t, \delta) \ : \  \sup_n  \ E\sum_{i = N_0(t) + 1}^\infty \|\tilde{A}^n(t)\|^8 u(i)  < \delta. 
\end{align}
\end{lem}
\begin{proof}
First notice that the assumptions lead to the existence of constant $K$ such that ($b^n, \sigma^n$ being the coefficients of SDE for $A^n$) \begin{equation} \begin{aligned} \label{usualpp} 
& |b_{i, x}^n(a)| \vee \|\sigma_{i,x}^n(a)\|_{\mathbb{R}^{2N}} \leq K(1 + |a_{i, x}|) \\
& |b_{i, y}^n(a)| \vee \|\sigma_{i,y}^n(a)\|_{\mathbb{R}^{2N}} \leq K(1 + |a_{i, y}|) \\
& |b_{i, z}^n(a)| \vee \|\sigma_{i,z}^n(a)\|_{\mathbb{R}^{2N}} \leq K(1 + \sum_{j = 1}^3 |a_{i, j}|).
\end{aligned} \end{equation}
Suppose $0 < s < t \leq T$, we have (remind $|\Lambda_n| = N$)  $$E\| \tilde{A}^n(t) - \tilde{A}^n(s)\|_S^8 = E \sum_{i = 1}^N \|A^n_i(t) - A^n_i (s)\|_{\mathbb{H}}^8 u(i)$$ $$= \sum_{i = 1}^N E (((A^n_{i,x}(t) - A^n_{i,x}(s))^2 + (A^n_{i,y}(t) - A^n_{i,y}(s))^2)^2 + (A^n_{i,z}(t) - A^n_{i,z}(s))^2)^2 u(i)$$
\begin{align} \label{predupravou} \lesssim \sum_{i=1}^N u(i) \Big( E(A^n_{i,x}(t) - A^n_{i,x}(s))^8 + E (A^n_{i,y}(t) - A^n_{i,y}(s))^8 + E (A^n_{i,z}(t) - A^n_{i,z}(s))^4 \Big).\end{align}
The $x$ term is now estimated using (\ref{usualpp}), Burkholder - Davis - Gundy and H\"{o}lder inequalities $$ E(A^n_{i,x}(t) - A^n_{i,x}(s))^8 = E\left(\int_s^t b_{i,x}^n (A^n(u)) du + \int_s^t \sigma_{i, x}^n (A^n(u)) dW(u)\right)^8$$ $$\lesssim |t-s|^7 E(\int_s^t |b^n_{i,x}(A^n(u))|^8 du) + |t-s|^3E(\int_s^t \|\sigma^n_{i,x}(A^n(u))\|^8 du)$$
$$\lesssim |t-s|^2 + |t-s| \int_s^t E |A^n_{i,x}(u)|^8 du  .$$
Similarly handling the $y$ and $z$ we get $$E(A^n_{i,y}(t) - A^n_{i,y}(s))^8 \lesssim |t-s|^2 + |t-s| \int_s^t E |A^n_{i,y}(u)|^8 du) $$ $$E(A^n_{i,z}(t) - A^n_{i,z}(s))^4 \lesssim |t-s|^2  + |t-s| \int_s^t \sum_{j=1} ^3E |A^n_{i,j}(u)|^4 du .$$ Individual terms we treat $$E |A^n_{i,x}(u)|^8 = E\left||a_{i,x}| + \int_0^u b^n_{i,x}(A^n(v))dv + \int_0^u \sigma_{i, x}^n(A^n(v)) dW(v)\right|^8$$ $$\lesssim |a_{i, x}|^8 + 1 + \int_0^u E|A^n_{i, x}(v)|^8 dv, $$ analogically one gets $$E |A^n_{i,y}(u)|^8 \lesssim |a_{i, y}|^8 + 1 + \int_0^u E|A^n_{i, y}(v)|^8 dv $$ $$E |A^n_{i,z}(u)|^4 \lesssim |a_{i, z}|^4 + 1 + \int_0^u \sum_{j=1}^3 E|A^n_{i, j}(v)|^4 dv.$$ Altogether we derived existence of some constant $K(T) > 0$ such that $$E|A_{i, x}^n(u)|^8 + E|A_{i, y}^n(u)|^8 + \sum_{j=1}^3 E|A_{i, j}^n(u)|^4 \leq K(T)(\|a_i\|_{\mathbb{H}}^8 + 1) $$ 
$$ + K(T) \int_0^u\Big( E|A_{i, x}^n(u)|^8 + E|A_{i, y}^n(u)|^8 + \sum_{j=1}^3 E|A_{i, j}^n(u)|^4 \Big) du.$$ Invoking the Gr\"{o}nwall's inequality we can deduce existence of some constant $K_1(T) > $ such that $\forall u \in [s, t]$ \begin{align} \label{krucialniodhad} E|A_{i, x}^n(u)|^8 + E|A_{i, y}^n(u)|^8 + \sum_{j=1}^3 E|A_{i, j}^n(u)|^4 \leq K_1(T)(1 + \|a_i\|_{\mathbb{H}}^8).  \end{align}
Hence $$E(A^n_{i,x}(t) - A^n_{i,x}(s))^8 + E (A^n_{i,y}(t) - A^n_{i,y}(s))^8 + E (A^n_{i,z}(t) - A^n_{i,z}(s))^4 $$ $$\lesssim |t-s|^2 + |t-s|^2K_1(T)(1 + \|a_i\|_{\mathbb{H}}).$$ Installing back to (\ref{predupravou}) we obtain thanks to \textbf{(A2)} and the fact that $a \in E$ the existence of some constants $L(T), C(T) > 0$ such that $$E\|\tilde{A}^n(t) - \tilde{A}^n(s)\|_S^8 \leq \sum_{i = 1}^N u(i)|t-s|^2 L(T)(1 + \|a_i\|_{\mathbb{H}}^8) $$ $$\leq C(T) |t-s|^2, $$ which we wanted to prove (\ref{spojitostodhad}). \\
To prove (\ref{kompaktodhad}) we simply utilize the key estimate (\ref{krucialniodhad}), which gives us $$E \sum_{N_0(t) + 1}^\infty \|A^n_i(t)\|_{\mathbb{H}}^8 u(i) \lesssim \sum_{N_0(t) + 1}^\infty u(i) K_1(t)(1 + \|a_i\|_{\mathbb{H}}^8), $$ 
therefore for given $\delta > 0$ it suffices to choose $N_0(t)$ such that the sum $\sum_{i = N_0(t) + 1}^\infty u(i)(1 + \|a_i\|_\mathbb{H}^8)$ is sufficiently small. 
\end{proof}
\begin{cor} \label{tesnostaproximace} Let $\tilde{A}^n$ as in Lemma \ref{momentoveodhady}. Then $P \circ (\tilde{A}^n)^{-1}, n \geq 1$ is tight sequence of measures in $\Omega$.
\end{cor}
\begin{proof}
The estimate (\ref{spojitostodhad}) implies according to Theorem \ref{stejnspoj} that equicontinuity condition is satisfied. Since boundedness is immediately implied by equicontinuity and boundedness at zero, to prove precompactness on a dense subset it remains to show by Lemma \ref{kompaktnostvE} that for given $\epsilon > 0$ \begin{align} \label{bodovakompaktnost}
P\left(\forall_{t \in \mathbb{Q} \cap (0, \infty)} \ \forall_{\delta \in (0, \infty \cap \mathbb{Q}} \ \exists_{N_0(t, \delta)} : \sum_{i = N_0(t, \delta) + 1}^\infty \|A^n_i(t)\|_\mathbb{H}^8 u(i) < \delta\right) > 1 - \epsilon.
\end{align} For any $\epsilon > 0$ and given fixed $t$ and $\delta$ application of Chebyshev inequality in conjunction with the estimate (\ref{kompaktodhad}) yields \begin{align*}
& P \left(\sum_{i = N_0(t, \delta) + 1}^\infty \|A^n_i(t)\|_\mathbb{H}^8 u(i) < \delta\right) = 1 - P \left(\sum_{i = N_0(t, \delta) + 1}^\infty \|A^n_i(t)\|_\mathbb{H}^8 u(i) \geq \delta \right) \\ & \geq 1 - \frac{E \sum_{i = N_0(t, \delta) + 1}^\infty \|A^n_i(t)\|_\mathbb{H}^8 u(i)}{\delta} > 1 -\epsilon.
\end{align*} Considering we have only countably many $\delta$'s and $t$'s, standard argument shows that (\ref{bodovakompaktnost}) is indeed fulfilled. 
\end{proof} 
\subsection{Solution to the Martingale problem} \label{Mart_problem_sekce}
Now we show that weak limit of sequence $\lbrace P \circ (\tilde{A}^n)^{-1} \rbrace$ can be used to construct martingale solution to the operator (\ref{infdimopera}).    \\
We let $A_t(w) = w(t)$, $w \in \Omega$ be the canonical process on $\Omega = C([0, \infty), S)$ with $\sigma$-algebra $\mathcal{F} = \sigma(w(s), s \geq 0)$, $\mathcal{F}_t = \sigma(w(s), 0 \leq s \leq t)$ denotes the usual filtration. We further introduce spaces $\Omega_n = C([0, \infty), (\mathbb{R}^3)^{\Lambda_n})$, $B^n_t(\omega_n) = \omega_n(t)$ the canonical process on $\Omega_n$ and the mappings \begin{align*} &\chi_n : (\mathbb{R}^3)^{\Lambda_n} \rightarrow S, \ i_n(a_1, \ldots, a_N) = (a_1, \ldots, a_N, 0_{i \in \mathbb{Z}^d \setminus \Lambda_n}) \\
& \psi_n : \Omega_n \rightarrow \Omega, \ \omega_n \rightarrow [t \rightarrow (\omega_n(t), 0_{i \in \mathbb{Z}^d \setminus \Lambda_n})]. \end{align*}
For given $a \in S$ we denote $A^{n, a}$ and $\tilde{A}^{n, a}$ the processes constructed in previous section to accentuate their dependence on $a$, i. e. $A^{n,a}$ is the solution to SDE with generator extending the (\ref{aproxopera}), $A^n(0) = \pi_{\Lambda_n} (a)$ and $\tilde{A}^n = (A^n, 0_{ i \in \mathbb{Z}^d \setminus \Lambda_n })$ . In addition we denote by $P^a$ the weak limit of measures $P \circ (\tilde{A}^{n,a})^{-1}$, that we just proved in Corollary \ref{tesnostaproximace} to exist. To simplify the notation we denote $\tilde{P}^a_n = P \circ (\tilde{A}^{n,a})^{-1}$ and $P^a_n = P \circ (A^{n,a})^{-1}$, the matching expectations will then be denoted $E^a, \tilde{E}^a_n,$ respectively $E^a_n$ . Notice that $\tilde{P}^a_n = P^a_n \circ \psi_n^{-1}$, as following calculation reveals : for $C \in \mathcal{F}$ $$\tilde{P}^a_n(C) = P(\tilde{A}^{n, a} (\cdot) \in C) = P((A^{n, a}, 0)(\cdot) \in C) = P(\psi_n(A^{n,a}) \in C)$$ $$= P^a_n \circ \psi_n^{-1}(C).$$
We introduce two family of functions. We say that $f \in C^{2, Cyl}_c(S)$, if there exists $\Phi_f \subset\subset \mathbb{Z}^d$ such that there is  $g \in C^2_c((\mathbb{R}^3)^{\Phi_f}, \mathbb{R})$ ($c$ stands for compactly supported) and $f(a) = g(\pi_{\Phi_f}(a))$, analogically $f \in C^{2, Cyl}(S)$, if such $g \in C^2((\mathbb{R}^3)^{\Phi_f}, \mathbb{R})$. With this notation we arrive at the following theorem.
\begin{thm}[Existence of solution to the martingale problem] \label{existence reseni mart problemu}
Let $a \in S$. Then there exists measure probability measure $P^a$ on $\Omega$ such that : \begin{align} \label{hmotavnule}
& P(A_0 = a) = 1 \\
& f(A_t) - f(A_0) - \int_s^t Lf(A_u) du \label{martingalovarovnost}
\end{align}
is $\mathcal{F}_t$-martingale under $P^a$ for any $f \in C^{2, Cyl}_c(S)$ and $\mathcal{F}_t$-local martingale under $P^a$ for any $f \in C^{2, Cyl}(S).$
\end{thm}
\begin{proof}
Define $P^a$ as above, so that we have $\tilde{P}^a_n \xrightarrow{w} P^a$. Then with the aid of Portmanteau theorem $$P(A_0 = a) = 1 - \sum_k P^a(\|A_0 - a\|_S > \frac{1}{k})$$ $$\geq 1 - \sum_k \liminf_n P(\|\tilde{A}^{n, a}(0) - a\|_S > \frac{1}{k}) = 1 - \sum_k 0, $$ thus we see that (\ref{hmotavnule}) is satisfied. Let $f \in C^{2, Cyl}_c(S)$ be given. To prove that (\ref{martingalovarovnost}) is martingale it suffices to prove by standard technique (see \cite[Lemma 3.1]{Seidler}) that for arbitrary $G \in C(C([0, s], S), [0, 1]), s < t$ \begin{align} \label{dukazmartingalu}
E^a\left[\left(f(A_t) - f(A_s) - \int_s^t Lf(A_u) du\right)G(\omega_\cdot)\right] = 0.
\end{align}
By weak convergence $\tilde{P}^a_n \xrightarrow{w} P^a$ the formula in (\ref{dukazmartingalu}) is a limit of \begin{align} \label{predprenosem} \tilde{E}^a_n\left[\left(f(A_t) - f(A_s) - \int_s^t Lf(A_u) du\right)G(\omega_\cdot)\right].
\end{align}
We compute \begin{equation} \begin{aligned} \label{komputace na prenos}
& \tilde{E}^a_n f(A_t(\omega)) = \tilde{E}_n^a f(\omega_t) = E^a_n f([\psi_n \omega_n]_t) = E^a_n (f \circ \chi_n)(B_t^n(\omega_n)) \\
& \tilde{E}^a_n G(\omega_\cdot) = E^a_n G((\psi_n \omega_n)_\cdot) = E^a_n (G \circ \psi_n)((\omega_n)_\cdot). \end{aligned} \end{equation}
Since $f$ is cylindrical the operator $L$ acting on $f$ in fact reduces to $L^f$, i. e. the operator $$L^f = \sum_{i \in \Phi_f} \mathcal{L}_{\lambda_i} + q_{x_i} X_i + q_{y_i} Y_i.$$ Consider that for $n$ large enough every point from $\Phi_f$ has all neighbours in $\Lambda_n$ and hence $L^f$ equals to  $L_n$ on $\Phi_f$, where $L_n$ is the operator corresponding to $A^n$ as defined in (\ref{aproxopera}). Then we adjust $$\tilde{E}^a_n \int_s^t L f(A_u) = E^a_n \int_s^t L^f f([\psi_n \omega_n]_u) = E^a_n \int_s^t L^f f(\chi_n (B^n_u(\omega_n))$$ $$= E^a_n \int_s^t L_n (f \circ \chi_n) (B^n_u(\omega_n)). $$ 
Altogether we found out that (\ref{predprenosem}) is equal to $$E^a_n\left[\left((f \circ \chi_n)(B_t^n) -  (f \circ \chi_n)(B_s^n) - \int_s^t L_n (f \circ \chi_n) (B^n_u) du \right)(G \circ \psi_n) ((\omega_n)_\cdot)\right], $$ but since we know that $P^a_n$ solves the martingale problem for $L_n$ on $\Omega_n$, this expression equals to zero and therefore also (\ref{dukazmartingalu}) is zero. \\
To deduce that for $f \in C^{2, Cyl}(S)$ (\ref{martingalovarovnost}) is local martingale, is the same as in finite dimension thanks to the cylindricity assumption.
\end{proof} 
\section{Uniqueness of approximating procedure}
In the previous section we only showed that our approximation scheme is tight, however now we show under additional assumptions that limit point is unique. 
To make the calculations as simple as possible (although still far from trivial) we distinguish specific approximation scheme related to the size of our interactions. Recall that $0 < r < \infty$ is the parameter of length of interactions for the functions $q$'s. We define boxes $$\Pi_n = \lbrace {i \in \mathbb{Z}^d : \ \max_{j \leq d} |i_j| \leq nr} \rbrace, N = |\Pi_n| = (2nr+1)^d.$$\\
We need to impose on the interactions additional assumption \textbf{(H2)} : \begin{itemize}
\item $\sup_{u \in (\mathbb{R}^3)^{(2r + 1)^d}} \sum_{j = 1}^{(2r+1)^d} |\frac{\partial q_{\cdot_i}}{\partial_j}(u) u_{\cdot_i} | +  |\frac{\partial q_{\cdot_i}}{\partial_j}(u)| \leq C, \ i \in \mathbb{Z}^d$
\end{itemize}
This assumption ensures that the equation for $A^n$ has globally Lipschitz drift. More precisely we need the following observation.
\begin{lem} \label{LipschitzchovaniAn}
Let $\Lambda_n \supset \Pi_{k+1}$ and we denote $b_k = (b_1, \ldots, b_K)$ (notice that this does not depend on $n$, since we assume $\Lambda_n \supset \Pi_{k+1}$) the first $K = |\Pi_k|$ coordinates of drift for the equation $$dA^n = b^n(A^n)dt + \sigma^n(A^n)dW_t,$$ also for an element $c^n \in (\mathbb{R}^3)^{\Lambda_n}$ we denote $c^n_k = (c^n_{1, x}, \ldots, c^n_{K, z})$. Then there exists constant $L > 0$ s. t. \begin{align} \label{lipschicdrift} \|b_k(a^n) - b_k(d^n)\|_{(\mathbb{R}^3)^{\Pi_k}}^2 \leq L \|a^n_{k+1} - d^n_{k+1}\|_{(\mathbb{R}^3)^{\Pi_{k+1}}}^2, \ \forall a^n, d^n \in (\mathbb{R}^3)^{\Lambda_n}.\end{align} $L$ is independent of $k, n$. 
\end{lem}
\begin{proof} The proof is elementary and follows from assumptions \textbf{(H1), (H2), (H3)}. The terms in the drift that complicate Lipschitz condition  - and force us to use $k+1$ in (\ref{lipschicdrift}) - are the ones containing $q_\cdot$'s, since they depend on all nearest $(2r+1)^d$ neighbours. As an example, how one obtains (\ref{lipschicdrift}) in these cases, we handle using the notation just introduced e. g. the term $q_y x$. Because of the finite range of our interactons  $q_{i,x} (\cdot) a_{i,x}^n$ is a smooth function of $(2r+1)^d$ variables for $i, 1 \leq i \leq K$, hence application of mean value theorem together with \textbf{(H2)} yields \begin{align*} & \left(q_{i,y} (a^n) a_{i,x}^n - q_{i,y}(d^n) d_{i,x}^n \right)^2 \leq \|\nabla q_{i,y}(\cdot)\cdot_{i,x}\|_{\infty} \|a^n_{r \supset i} - d^n_{r \supset i}\|_{(\mathbb{R}^3)^{(2r+1)^d}} \\ & \leq C \|a^n_{r \supset i} - d^n_{r \supset i}\|_{(\mathbb{R}^3)^{(2r+1)^d}},
\end{align*} where we denoted $r \supset i = \lbrace j \in \mathbb{Z}^d : |j - i|_{\max} \leq r \rbrace$. \\ We then take into account that every point $i \in \mathbb{Z}^d$ has the same finite fixed amount of neighbours. Hence handling the other terms in the obvious way, we indeed arrive at the existence of some $L> 0$ such that $$\|b_k(a^n) - b_k(d^n)\|_{(\mathbb{R}^3)^{\Pi_k}}^2 \leq L \|a^n_{k+1} - d^n_{k+1}\|_{(\mathbb{R}^3)^{\Pi_{k+1}}}^2, \ \forall a^n, d^n \in (\mathbb{R}^3)^{\Lambda_n}.$$
\end{proof}
\noindent In addition we need to restrict our class of starting points $a \in S$, so that the space includes only configurations that does not grow too fast, i. e. we introduce \textbf{(H6)} : \begin{itemize} 
\item $\exists \ \delta \in (0, 1)$ $\exists K > 0$ s. t. $$u(j) \geq \frac{K}{i!^{1 - \delta}} \  \ j \in \Pi_i \setminus \Pi_{i-1}, \ i \in \mathbb{N}.$$
\end{itemize}
Comparing this assumption with the restrictions on weights that Da Prato and Zabczyk need to impose \cite[pp. 10]{Zabczyk}, we see that our conditions include faster growing configurations. \\
The key to proofs in this section are two technical Lemmas about behaviour of solutions $A^n$ to the SDE's related to the operator $L_n$. If we take some fixed given set $\Gamma \subset \mathbb{Z}^d$ and two supersets $\Gamma_n, \Gamma_k \supset \Gamma$, such that we have corresponding solutions $A^n, A^k$ of SDE's on $(\mathbb{R}^3)^{\Gamma_n}$ resp. $(\mathbb{R}^3)^{\Gamma_k}$, then we cannot claim that $(A^n_i)_{i \in \Gamma}$ and $(A^k_i)_{\i \in \Gamma}$ have the same distribution, because we have to redefine the interaction functions at the boundary of the sets $\Gamma_n, \Gamma_k$, and hence $(A^n_i)_{i \in \Gamma}$ and $(A^k_i)_{i \in \Gamma}$ differ as they depend on all $A^n$ resp. $A^k$ via interactions. Therefore we can never have precise equality, even though the part of equations on $(\mathbb{R}^3)^\Gamma$ will have the same coefficients, once both $\Gamma_n$ and $\Gamma_k$ includes all neighbours of $\Gamma$. Nevertheless one intuitively would expect, that the further we are from boundary, the smaller the effect of redefining should be on $\Gamma$. Next Lemma formalizes and justifies this intuition. Then we can also interpret technical assumption \textbf{(H6)} by saying, that the effect of redefining at the boundary will be small, provided we do not start from very fast growing initial configurations. \\
In all what follows in this chapter we assume conditions \textbf{(H1) - (H4), (H6)}.
\begin{lem} \label{asymptotickarovnost} Let $a \in S$ and $\Pi_k$ be defined as above. Suppose we have two exhausting sequences $\lbrace \Lambda_l \rbrace, \lbrace \Lambda_m \rbrace$ of $\mathbb{Z}^d$ and correspondingly two sequences of processes $\lbrace A^{m,a} \rbrace, \lbrace A^{l,a} \rbrace$. We denote by $A^{m, a}_k$ the part of $A^{m, a}$ that lives on $(\mathbb{R}^3)^{\Pi_k}$, i. e. $A^{n, a}_k = (A^n_{1, x}, \ldots, A^n_{K, z})$. For any $\epsilon > 0$ and $T > 0$ there exists $N > 0$ such that for any $l, m \geq N$ \begin{align} \label{konvergencesamotneposl}
E\sup_{t \in [0, T]} \|A^{l, a}_k(t) - A^{m,a}_k(t)\|_{(\mathbb{R}^3)^{\Pi_k}}^2 \leq \epsilon.
\end{align}
\end{lem}
\begin{proof}
We will be little imprecise and write $a_k = (a_{1, x}, \ldots, a_{K, z})$ for the restriction of $a$ to $(\mathbb{R}^3)^{\Pi_k}$, in order to not overload the notation we also write $a_j = (a_{j, x}, a_{j, y}, a_{j, z})$ when $j \in \mathbb{Z}^d$. Also when dealing with the norms on spaces $\mathbb{R}^n$ for different $n$ we omit the index in the norm, as it should not lead not confusion and instead enhance readability. Using the Lemma \ref{LipschitzchovaniAn} and routine calculations for SDE's we compute \begin{align*} &E\sup_{t \in [0, T]} \|A^{l, a}_k(t) - A^{m,a}_k(t)\|^2 \\ & \lesssim  E\sup_{t \in [0, T]} \left( \left\| \int_0^t b_k(A^{l, a}) - b_k(A^{m,a}) ds \right\|^2 + \left\| \int_0^t \sigma_k(A^{l,a}) - \sigma_k(A^{m,a}) dW_s \right\|^2 \right) \\ & \lesssim T \left( E \int_0^T \|b_k(A^{l, a}) - b_k(A^{m, a})\|^2 ds + E \int_0^T \|\sigma_k(A^{l,a}) - \sigma_k (A^{m, a})\|^2 ds \right) \\ & \lesssim T\left( \int_0^T E \|A^{l, a}_{k+1}(t_1) - A^{m, a}_{k+1}(t_1)\|^2 dt_1 \right).     \end{align*}
Therefore we obtained the existence of constant $C > 0$ so that $$E\sup_{t \in [0, T]} \|A^{l, a}_k(t) - A^{m,a}_k(t)\|^2 \leq CT\int_0^TE \|A^{l, a}_{k+1}(t_1) - A^{m, a}_{k+1}(t_1)\|^2 dt_1.$$ Assuming $l, m$ large enough so we can repeat the procedure above, we get \begin{equation} \begin{aligned} \label{iterovanyodhad}
& E \|A^{l, a}_{k+1}(t_1) - A^{m,a}_{k+1}(t_1)\|^2 \leq C t_1 \int_0^{t_1} E \|A^{l, a}_{k+2}(t_2) - A^{m,a}_{k+2}(t_2)\|^2dt_2 \\
& \cdots \leq C^{n-1} t_1 \int_0^{t_1} t_2 \int_0^{t_2} \cdots \int_0^{t_{n-1}} E \|A^{l, a}_{k+n}(t_{n}) - A^{m,a}_{k+n}(t_{n})\|^2 dt_n \ldots dt_1. 
\end{aligned} \end{equation}
Thanks to the Linear growth of coefficients of our SDE (\ref{usualpp}), there is some $K_T > 0$ such that $$E \|A^{l, a}_{k+n}(t_{n}) - A^{m,a}_{k+n}(t_{n})\|^2 \leq K_T(1 + \|a_{n+k}\|^2).$$
Using this and then calculating the iterated integrals, we obtain from (\ref{iterovanyodhad}) the estimate $$E\sup_{t \in [0, T]} \|A^{l, a}_k(t) - A^{m,a}_k(t)\|^2 \leq \frac{(CT^2)^n}{(2n-1)!!} K_T(1 + \|a_{n+k}\|^2),$$
where $(2n-1)!! = (2n-1) \cdot (2n-3) \cdots 3\cdot 1$ denotes the odd (double) factorial. Using the obvious $$\|a_{n+k}\|^2_{(\mathbb{R}^3)^{\Pi_{n+k}}} \leq \sum_{j = 1}^{(2(n+k)r + 1)^d} 3 + \|a_j\|^8_\mathbb{H},$$ we need to prove only $$\lim_{n \rightarrow \infty} \frac{L^n}{n!}\sum_{j = 1}^{(2(n+k)r + 1)^d} (1 + \|a_j\|^8_\mathbb{H}) = 0 $$ for arbitrary constant $L > 0$. Clearly it suffices to show \begin{align} \label{chcemespocitat} \lim_n \frac{\sum_{j = 1}^{(2(n+k)r + 1)^d} \|a_j\|^8_\mathbb{H}}{n!^{1 - \frac{\delta}{2}}} = 0, \end{align} where $\delta$ is from the assumption \textbf{(H5)}. We compute using the \textbf{(H5)} and trivial $\|a_j\|_\mathbb{H}^8 u(j) \leq \|a\|_S^8$ $$
\lim_n  \frac{\sum_{j \in \Lambda_{n+k+1} \setminus \Lambda_{n+k}} \|a_j\|^8_{\mathbb{H}}}{n!^{1 - \frac{\delta}{2}}((n+1)^{1-\frac{\delta}{2}}-1)}$$ 
\begin{align} \label{vypocetlimity}
\leq \frac{\|a\|_S^8}{K} \lim_n \frac{(2(n+k+1)r + 1)^d - (2(n+k)r + 1)^d)(n+k+1)!^{1 - \delta}}{n!^{1 - \frac{\delta}{2}}} = 0.
\end{align} The fact that (\ref{vypocetlimity}) implies (\ref{chcemespocitat}) is well known as Stolz - Ces\`aro Theorem (see \cite[pp.85]{Stolz}).
\end{proof}
\begin{lem} \label{stejnaspojitostreseni} Let $k \in \mathbb{N}$,  $a \in S$ and $t > 0$ be given. Let $A^{m,a}$ be approximating sequence defined with respect to exhausting boxes $\Pi_m$. For any $\epsilon > 0$ there exists $\eta > 0$ such that $\forall m \geq k$ \begin{align} \label{spojitostdukaz}
\|b - a\|_S < \eta \Longrightarrow E\|A^{m, a}_k(t) - A^{m, b}_k(t)\|^2 < \epsilon.
\end{align}
\end{lem}
\begin{proof}
Since we know that our SDE has continuous dependence on initial condition, the Lemma is nontrivial only for infinite number of $m$ and hence we concentrate in our computations on large $m$. Again for simplification we will not write the index to the norms through computations. Similarly to the last Lemma we get for some constants $C > 0$ and $K_t > 0$ (to make last sum meaningful let us formally define $(-1)!! = 1$)
\begin{align*}
& E\|A^{m, a}_k(t) - A^{m, b}_k(t)\|^2 \leq C \|a_k - b_k\|^2 + Ct \int_0^{t} E\|A^{m, a}_{k+1}(t_1) - A^{m, b}_{k+2}(t_1)\|^2 dt_1 \\
& \leq C\|a_k - b_k\|^2 + C^2t^2 \|a_{k+1} - b_{k+1}\|^2 \\  
& + Ct\int_0^t Ct_1 \int_0^t E\|A^{m, a}_{k+2}(t_2) - A^{m, b}_{k+2}(t_2)\|^2 dt_2 dt_1 \\
& \leq C\|a_k - b_k\|^2 + C^2t^2 \|a_{k+1} - b_{k+1}\|^2 + \cdots + \frac{C^{n}t^{2n-2}}{(2n-3)!!}\|a_{k+n-1} - b_{k+n-1}\|^2 \\
& +  E \sup_{0 \leq s \leq t} \|A^{m, a}_{k+n}(s) - A^{m, b}_{k+n}(s)\|^2 \frac{(Ct^2)^n}{(2n-1)!!} \\
& \leq \sum_{j = 1}^{n} \frac{C^j t^{2j-2} \|a_{k+j-1} - b_{k+j-1}\|^2}{(2j-3)!!} + K_t \left(\sum_{i \in \Lambda_{k+n}} 3 + \|a_i\|^8_\mathbb{H} + \|b_i\|^8_\mathbb{H} \right)  \frac{(Ct^2)^n}{(2n-1)!!}.
\end{align*}
Same calculations like in Lemma \ref{asymptotickarovnost} together with Stolz - Ces\`aro Theorem gives \begin{align} \label{druhy clen do nuly}
\lim_{n \rightarrow \infty} K_t \left(\sum_{i \in \Lambda_{k+n}} 3 + \|a_i\|^8_\mathbb{H} + \|b_i\|^8_\mathbb{H} \right)  \frac{(Ct^2)^n}{(2n-1)!!} = 0.
\end{align}
Because $$\lim_{n \rightarrow \infty} \frac{C^n t^{2n-2} n^l}{((2n-1)!!)^{\frac{\delta}{2}}} = 0$$ for $l > 1$, we obtain using previously established convergence results that \begin{align} \label{druhyclenkonverguje}
\sum_{j = 1}^{\infty} \frac{C^j t^{2j-2} \|a_{k+j-1} - b_{k+j-1}\|^2}{(2j-3)!!} < +\infty.
\end{align} Therefore combining (\ref{druhy clen do nuly}) and (\ref{druhyclenkonverguje}) for given $\epsilon > 0$ we can choose $N \in \mathbb{N}$ such that \begin{align*} 
& \sum_{j = N}^{\infty} \frac{C^j t^{2j-2} \|a_{k+j-1} - b_{k+j-1}\|^2}{(2j-3)!!} \\
&  + \sup_{j \geq N} K_t \left(\sum_{i \in \Lambda_{k+j}} 3 + \|a_i\|_\mathbb{H}^8 + \|b_i\|_\mathbb{H}^8 \right)  \frac{(Ct^2)^j}{(2j-1)!!} < \frac{\epsilon}{2}. \end{align*}
For the first $N - 1$ terms we can choose $\eta > 0$ in (\ref{spojitostdukaz}) thanks to the continuous dependence on parameters for the $A^{m,a}$ in such way that \begin{align*} 
& \sum_{j = 1}^{N-1} \frac{C^j t^{2j-2} \|a_{k+j-1} - b_{k+j-1}\|^2}{(2j-3)!!} \\
&  + \sup_{j \leq N-1} E \sup_{0 \leq s \leq t} \|A^{m, a}_{k+j}(s) - A^{m, b}_{k+j}(s)\|^2 \frac{(Ct^2)^j}{(2j-1)!!}   < \frac{\epsilon}{2}, \end{align*} and the Lemma is established.
\end{proof}
\noindent The first crucial property that follows from Lemma \ref{asymptotickarovnost}) is  independence of the limit measure $P^a$ on the choice of convergent subsequence. By the well known properties of weak convergence this implies that the sequence $\lbrace \tilde{P}^a_n \rbrace$ itself weakly converges. In addition this limit doesn't depend on the choice of approximating sequence $\Lambda_n$.  
\begin{thm} \label{jednoznacnost limity} Let $\tilde{A}^{m, a}, \tilde{A}^{n,a}$ be the sequences of approximating processes on $\Omega$, $a \in S$. Then there exists probability measure $P^a$ on $\Omega$ such that $$\lim_{m \rightarrow \infty} P \circ (\tilde{A}^{m, a})^{-1} = \lim_{l \rightarrow \infty} P \circ (\tilde{A}^{l, a})^{-1} = P^a.$$ 
\end{thm}
\begin{proof}
By Corollary \ref{tesnostaproximace} we know that any two such sequences has weakly convergent subsequence. So it remains to show that the limit point is the same for any two weakly convergent subsequences (to simplify notation we call the convergent subsequences again $m$ and $l$) $\lbrace P \circ (\tilde{A}^{l, a})^{-1} \rbrace, \lbrace P \circ (\tilde{A}^{m, a})^{-1} \rbrace$. To prove this it clearly suffices to show that for any $f \in C_b(\Omega)$ \begin{align} \label{dokaze jednoznacnost}
\lim_l E f(\tilde{A}^{l, a}(\cdot)) = \lim_m E f(\tilde{A}^{m, a} (\cdot)).
\end{align}
First let $f \in C^{Cyl}_{b, Lip} (\Omega)$, i. e. there exists $k \in \mathbb{N}$ and $g \in C_{b, Lip} (\Omega_{\Pi_k})$ such that $f(\omega) = g((\pi_{\Pi_k}\omega)_\cdot)$, $\Omega_{\Pi_k} = C([0,\infty), (\mathbb{R}^3)^{\Pi_k})$ and $g$ is Lipschitz, that is there exists constant $L > 0$ s. t. $$|g((\omega_k)_\cdot) - g((\tilde{\omega}_k)_\cdot)| \leq \|\omega_k - \tilde{\omega}_k\|_{\Omega_{\Pi_k}} \ \forall \omega_k, \tilde{\omega}_k \in \Omega_{\Pi_k}.$$ Then we get for $m, l$ large enough
\begin{align*} &|E f(\tilde{A}^{l, a}(\cdot)) - E f(\tilde{A}^{m, a} (\cdot))|^2 = |E g(A^{l, a}_k(\cdot)) -  E g(A_k^{m, a} (\cdot))|^2 \\
& \leq E|g(A^{l, a}_k(\cdot)) -  g(A_k^{m, a} (\cdot))|^2 \leq E\|A^{l, a}_k(\cdot) -  A_k^{m, a} (\cdot)\|^2,
\end{align*} hence Lemma \ref{asymptotickarovnost} implies (\ref{dokaze jednoznacnost}) holds for $f \in C^{Cyl}_{b, Lip} (\Omega).$ \\
Next let $f \in C^{Cyl}_b(\Omega)$, then there eixsts bounded sequence $f_n \in C^{Cyl}_{b, Lip}(\Omega)$ such that $f_n \rightarrow f$. Finally for $f \in C_b(\Omega)$ consider cylindrical approximation by $\lbrace f_n \rbrace$, that is $f_n(\omega_\cdot) = f((\pi_{\Pi_n} \omega)_\cdot)$ and the result follows by Lebesgue Theorem. 
\end{proof}
\subsection{Markov property}
\noindent To translate Lemma \ref{stejnaspojitostreseni} into desired properties, we need to recall result about strengthening of weak convergence. Its proof follows immediately from Skorokhod representation theorem (see also \cite[pp. 168]{Rogers}). 
\begin{lem} \label{slabakonvergence}
Let $P$ be a Polish space and $\mu_n, \mu$ probability measures on $P$. Suppose $\mu_n \xrightarrow{w} \mu$. Let $f_n, f \in C(P)$ such that $f_n$ are uniformly bounded and \begin{align} \label{zesilenislabek} x_n \rightarrow x \ \text{in} \ P  \Longrightarrow f_n(x_n)) \rightarrow f(x). \end{align} Then $\mu_n f_n \rightarrow \mu f. $   
\end{lem}
\noindent With this Lemma in hand we can now show that canonical process on $\Omega$ is a genuine Markov process under the measures $P^a$.
\begin{thm} \label{Markovsky proces} Let $A_t(w)$ be canonical process on $\Omega = C([0, \infty), S)$ and $P^a$ the unique limiting measure produced by Corollary \ref{tesnostaproximace}. $(A_t, P^a)$ is then a Markov process.  
\end{thm} 
\begin{proof}
Denote $\mathcal{S}$ the $\sigma$-algebra on $S$. We need to show these two properties \begin{align} \label{meritelnost}
& \textbf{(I)} \ a \rightarrow P^a (A_t \in C) \ \text{is measurable for any} \ C \in \mathcal{S}  \\ \label{skladacivlastnost}
& \textbf{(II)} \ P^a(A_{s+t} \in C | \mathcal{F}_s) = \phi(A_s), \ \phi(\cdot) = P^\cdot (A_t \in B), \ \forall \ C \in \mathcal{S}, 0 \leq s \leq t. \end{align}
To prove (\ref{meritelnost}) we show that $a \rightarrow E^a f(A(t))$ is continuous function for any $f \in C^{Cyl}_{b, Lip} (S)$, the measurability for general $f \in C_b(S)$ will then follow through same procedure as in Theorem \ref{jednoznacnost limity}. By the uniqueness just proved, we can consider approximation $\lbrace A^n \rbrace$ living on the boxes $\Pi_n$. So let $f(a) = g(\pi_{\Pi_k}(a))$, we then calculate \begin{align*}
& |E^a f(A_t) - E^b f(A_t)|^2 = |\lim_n E[f(\tilde{A}^{n, a}(t)) - f(\tilde{A}^{n, b}(t))]|^2 \\
& \leq \limsup_n E|g(A^{n, a}_k(t)) - g(A^{n, a}_k(t))|^2 \leq \limsup_n \|A^{n, a}_k(t) - A^{n, b}_k(t)\|^2.
\end{align*} From Lemma \ref{stejnaspojitostreseni} we derive that this estimate establishes the desired continuity. \\
To prove (\ref{skladacivlastnost}) one strives to establish  $\forall f \in C_b(S)$ \begin{align} \label{Markovskostpomocispojfci} E^a[f(A_{s+t})|\mathcal{F}_s] = E^{A_s} f(A_t). \end{align} If we denote $\varphi(\cdot) = E^\cdot[f(A_t)]$, this then means - for any $C \in \mathcal{F}_s$ $$\int_C f(A_{s+t}) dP^a = \int_C \varphi(A_s) dP^a. $$ We consider first $f \in C^{Cyl}_{b, Lip}(S)$, then we know from the first part of the proof that $\varphi(\cdot)$ is continuous. By approximation this reduces to necessity of demonstrating \begin{align} \label{Markovskostjinak}
E^a[f(A_{s+t})h(\omega_\cdot)] = E^a \varphi(A_s)h(\omega_\cdot), 
\end{align} where $h$ is arbitrary, but fixed continuous bounded $\mathcal{F}_s$ - measurable function. By weak convergence $\tilde{P}^a_n \rightarrow P^a$ the left side of (\ref{Markovskostjinak}) is a limit of (the same calculations as we made in the proof of Theorem \ref{existence reseni mart problemu} are hidden there) $$\tilde{E}^{a}_n [f(A_{s+t}) h(\omega_\cdot)] = E^a_n[(f \circ \chi_n) (B^n_{s+t}) (h \circ \psi_n)((\omega_n)_\cdot)]. $$ The finite dimensional result, i. e. the fact that $P^a_n$ solves the martingale problem on $\Omega_n$, tells us that $$E^a_n[(f \circ \chi_n) (B^n_{s+t}) (h \circ \psi_n) ((\omega_n)_\cdot)] = E^a_n[\varphi^n (\chi_n (B^n_s)) (h \circ \psi_n)((\omega_n)_\cdot)], $$ if $\varphi^n(\chi_n(B^n_s)) = E^{\chi_n(B_s^n)}_n[(f \circ \chi_n)(B_t^n)].$ We observe that $$\varphi(a) = E^a f(A_t) = \lim_n \tilde{E}^a_n f(A_t) = \lim_n E^a_n [(f \circ \chi_n) (B_t^n)], $$ hence (\ref{Markovskostjinak}) will established using Lemma \ref{slabakonvergence}, provided we can prove the implication \begin{align} \label{kzesilenislabe} a_n \rightarrow a \ \text{in} \ S \Longrightarrow \tilde{E}^{a_n}_n[f(A_t)] \rightarrow E^a[f(A_t)]. 
\end{align}
For given $\epsilon > 0$ we find $N$ from weak convergence such that $$|E^a[f(A_t)] - \tilde{E}^a_n [f(A_t)]| < \frac{\epsilon}{2} \ \forall n \geq N.$$
Like in the first part we also have the estimate $$|\tilde{E}^{a_n}_n[f(A_t)] - \tilde{E}^a_n [f(A_t)]|^2 \leq \limsup_n  \|A^{n, a_n}_k(t) - A^{n, a}_k(t)\|^2, $$ so Lemma \ref{stejnaspojitostreseni} implies we can find $\delta > 0, \tilde{N} \in \mathbb{N}$ such that $$\|a - a_n\| < \delta \ \Longrightarrow  |E^a[f(A_t)] - \tilde{E}^{a_n}_n [f(A_t)]| < \epsilon \ \forall n \geq \tilde{N}.$$ Therefore from Lemma \ref{slabakonvergence} we conclude that (\ref{Markovskostjinak}) holds for $f \in C^{Cyl}_{b, Lip}(S)$. We infer the  validity of (\ref{Markovskostpomocispojfci}) for general $f \in C_b(S)$ by routine approximation procedure. 
\end{proof}
\noindent This result gives us that if we set $P_t (a, C) = P^a \lbrace A_t \in C \rbrace$, then $P_t$ is a genuine transition probability function and $P_t f(a) = E^a f(A(t))$ is the Markov semigroup acting on all $f \in \mathcal{B}_b(S)$ (Borel bounded functions) satisfying the Chapman - Kolmogorov equality \cite[chap. I]{Blumenthal}. 
\section{Existence of invariant measure for the semigroup} \label{ergodickakapitola}
We now derive the tightness of measures $\lbrace \nu_n \rbrace$ and consequently show that any limit point is an invariant measure for the semigroup. \\
We need to enlarge our space $S$ to assure that it can accommodate invariant measure. The assumption we need in our case turns out to be \textbf{(H5)} : \begin{itemize}
\item $\exists v(i) > 0, i \in \mathbb{Z}^d, \sum_i v(i) < +\infty, \sum_i \frac{u(i)}{v(i)} < +\infty$. \end{itemize}
For the remainder of the article let us work under assumptions \textbf{(H1) - (H6)}. 
\begin{thm} \label{tesnost inv mer} The sequence of measures $\lbrace \nu_n \rbrace$ is tight.
\end{thm}
\begin{proof}
We want to show that for given $\epsilon > 0$ there is compact set $K_\epsilon$ in $S$ such that $\forall n \in \mathbb{N}$ one has $\nu_n(K_\epsilon) \geq 1 - \epsilon$.\\
Let us recall the estimate (\ref{obecnyLjapunov}) \begin{align} \label{pomocna_rovnost}
L_n W^2_n (a_n) \leq -c W^2_n(a_n) + C. 
\end{align}Remind that $\nu_n = \mu_n \circ i_n^{-1}$ and $\mu_n$ is an invariant measure. Hence we have the equality \begin{align} \label{rovnost_invmira} \mu_n(L_n W^2_n) = 0.\end{align} Clearly $$\mu_n(L_n W^2_n) = \mu_n(L_n W^2_n I_{L_n W^2_n > 0}) + \mu_n(L_n W^2_n I_{L_n W^2_n \leq 0}), $$ and from (\ref{pomocna_rovnost}) it follows that $\mu_n(L_n W^2_n I_{L_n W^2_n > 0}) \leq C$, so that \begin{align} \label{sporLjapunov}
\mu_n(L_n W^2_n I_{L_n W^2_n \leq 0}) \geq -C.
\end{align} Notice that in our notation it holds $$W^2_n(a_n) = \sum_{i = 1}^N V^2_i(a_i) v(i) = \sum_{i=1}^N v(i)\|a_i\|_{\mathbb{H}}^8.$$ For given $\epsilon >0$ we define the set $K_\epsilon$ as $$K_\epsilon = \left\lbrace a \in S: i \in \mathbb{Z}^d : \|a_i\|^8_{\mathbb{H}} u(i) \leq u(i)\left(\frac{C + 1}{c \epsilon v(i)} + \frac{C}{cv(i)} \right)\right\rbrace.$$ Thanks to the assumption \textbf{(H5)} this set is compact in $S$ according to Lemma \ref{kompaktnostvE}. We calculate $$\nu_n(K_\epsilon^C) = \mu_n(\psi_n(K_\epsilon^C))= \mu_n\left(b \in (\mathbb{R}^3)^{\Lambda_n} : \ \exists i \in \Lambda_n : \|b_i\|_\mathbb{H}^8 > \frac{C + 1}{c\epsilon v(i)} + \frac{C}{cv(i)}\right).$$ Hence for $a_n \in \psi_n(K_\epsilon^C)$ we have$$L_nW^2_n(a_n) \leq -c\left( \frac{C+1}{c \epsilon v(i)}+ \frac{C}{cv(i)}\right) v(i) + C \leq -\frac{C+1}{\epsilon}.$$ Therefore if $\nu_n(K_\epsilon^C) > \epsilon$ would hold, we would get the contradiction with (\ref{sporLjapunov}).
\end{proof}
\begin{thm} \label{Invariantni mira} There exists an invariant measure for the Markov process $(A, P^a)$ from Theorem \ref{Markovsky proces}.
\end{thm}
\begin{proof}
We fix some weakly convergent sequence of measures $\lbrace \nu_n \rbrace$ and its limit point $\nu$.  
To show that $\nu$ is invariant, we prove that for any $f \in C_b(S)$  \begin{align} \label{invariance nu} 
\int_S P_t f(a) d\nu(a) = \int_S f(a) d\nu(a). \end{align}
We show (\ref{invariance nu}) holds for $f \in C^{Cyl}_{b, Lip}(S)$, the general case then follows easily by approximation as before. Recall that $\nu_n = \mu_n \circ \chi_n^{-1}$ and that $\mu_n$ is the invariant measure for process $A^n$ on $(\mathbb{R}^3)^{\Pi_n}$, so that the equality $$\int_{(\mathbb{R}^3)^{\Pi_n}} E^{i_n(e_n)}_n h(B_t^n) d\mu_n(e_n) = \int_{(\mathbb{R}^3)^{\Pi_n}} h(e_n) d\mu_n(e_n) \ \ \forall h \in C_b((\mathbb{R}^3)^{\Lambda_n})$$ holds. Remembering the calculations (\ref{komputace na prenos}) we compute 
\begin{align*}
& \int_S f(a) d\nu(a) = \lim_n \int_S f(a) d\nu_n(a) = \lim_n \int_{(\mathbb{R}^3)^{\Pi_n}} (f \circ \chi_n) (a_n) d\mu_n(a_n) \\
& = \lim_n \int_{(\mathbb{R}^3)^{\Pi_n}} E^{\chi_n(a_n)}_n (f \circ \chi_n)(B_t^n) d\mu_n(a_n) = \lim_n \int_S E^{a}_n (f \circ \chi_n)(B_t^n) d\nu_n(a) \\
& = \lim_n \int_S \tilde{E}^a_n f(A_t) d\nu_n(a) \stackrel{?}{=} \int_S E^a f(A_t) \nu(a) = \int_S P_tf(a) d\nu(a).
\end{align*} We can erase the question mark using the Lemma \ref{slabakonvergence} with exactly the same line of reasoning that was required for the proof of (\ref{skladacivlastnost}) in previous Theorem. 
\end{proof}
\section{Examples of other operators}
We list some other relevant examples, that can be handled using our strategy without any additional difficulty : \begin{itemize}
\item Of course the elliptic case lies naturally within our framework. Take Euclidean space $\mathbb{R}^3$ with standard Laplacian $\Delta$, $D = x \partial_x + y \partial_y + z \partial_z$, $X = \partial_x$ (etc. for $Y$, $Z$), $\mathcal{L}_{\lambda} = \Delta - \lambda D$  and consider operator $$L = \sum_{i \in \mathbb{Z}^d} \mathcal{L}_{\lambda_i} + q_{i, x} X_i + q_{i, y} Y_i + q_{i, z} Z_i$$ acting on $(\mathbb{R}^3)^{\mathbb{Z}^d}$. Lyapunov function here can be chosen just $x^{2k} + y^{2k} + z^{2k}$, for $k = 2$ we get the same tightness as we had in Corollary \ref{tesnostaproximace}. 
\item The Grushin plane \cite{Grusin} : Take $\mathbb{R}^2$ as the basic space and consider vector fields $X = \partial_x, Y = -x \partial_y$. $D$ is given by $D = x \partial_x + y \partial_y$ and operator $$L = \sum_{i \in \mathbb{Z}^d} X^2_i + Y^2_i - \lambda_i D_i + q_{i,x}X_i + q_{i,y} Y_i$$ on $(\mathbb{R}^2)^{\mathbb{Z}^d}$. For the Lyapunov function works $V = x^{4k} + y^{2k}$, the tightness (\ref{tesnostaproximace}) works again for $k = 2$. The $\sigma$ and $u$ in Girsanov theorem to simplify the control problem can be chosen in the following way $$\sigma = \begin{pmatrix}
\sqrt{2} & 0 \\ 0 & \sqrt{2}x
\end{pmatrix} \ \ u = \begin{pmatrix}
\frac{q_x}{\sqrt{2}} \\ \frac{q_y}{x}
\end{pmatrix}. $$Then we have $$\sigma u = b - \tilde{b} = (-\lambda x, - \lambda y).$$
\item We cannot quite handle the example of Martinet distribution as in \cite{ZegarN}. Take $\mathbb{R}^3$ and let $X = \partial_x - y^2 \partial_z$, $Y = y \partial_y$. The problem that arises lies in the nonlinear term in $z$-axis. We can not hope for our strategy to be successful, as in the last section definitely linear growth together with strong Lipschitz condition is required. But at least the finite dimensional case is almost conquered by our methods - If one puts $D = x \partial_x + y \partial_y + z \partial_z$ and consider $$L = X^2 + Y^2 - \lambda D + q_x X + q_y Y$$ as operator on $\mathbb{R}^3$, then the SDE corresponding to this operator has coefficients $$b = (q_x - \lambda x, q_y - \lambda y, -\lambda z - q_x y^2), \ \sigma = \begin{pmatrix} \sqrt{2} & 0 & 0 \\ 0 & \sqrt{2} & 0 \\ 0 & 0 & \sqrt{2}y^2
\end{pmatrix}.$$
Due to nonlinearities, not even global existence of process is a priori clear. However, if we set $V_k = x^{2k} + y^{6k} + z^{2k}$, we calculate that $V_k$ is the Lyapunov function giving global existence and invariant measure. The smoothness of density holds from Theorem \ref{HormavPrav} as well. However to our best knowledge, we are unable to investigate the irreducibility of the process.  
\end{itemize} 
In general we can say, that our strategy is successful whenever we can establish finite dimensional results as in (\ref{konecnedimvysledky}) with Lyapunov function, that will enable us to construct the diffusion using tightness arguments as in chapter three. To finish the strategy with desired results, it is then essential that we can impose on the interaction such constraints that leads to the conditions of type (\ref{LipschitzchovaniAn}). \\

\begin{small}
\noindent \textbf{Acknowledgement.} The author feels the necessity of expressing sincerest gratitude towards dr Jan Seidler, who provided him with necessary feedback throughout the whole work and was willing to sacrifice lot of his time to the discussions about the arising issues. A lot is owed to Professor Boguslaw Zegarlinski for his generous support and fruitful proposal to study the topic. Special thanks goes to diligent anonymous referee, who contributed by many helpful suggestions and did his best to increase the comprehensibility of work, and to dr Jan Swart and dr Laurent Miclo, who pointed out the error in the previous version.
\end{small}
\begin{footnotesize}

\end{footnotesize}
\end{document}